\documentclass[11pt,reqno]{amsart}
\usepackage[T1]{fontenc}
\usepackage[utf8]{inputenc}
\usepackage{lmodern}
\usepackage{amsmath,amssymb,mathtools,mathrsfs}
\usepackage{microtype}
\usepackage{tikz-cd}
\usepackage[margin=1.05in]{geometry}
\usepackage{enumitem}
\usepackage[dvipsnames]{xcolor}
\definecolor{darkblue}{rgb}{0,0,0.6}
\usepackage[ocgcolorlinks]{hyperref}
\hypersetup{colorlinks=true,citecolor=darkblue,linkcolor=darkblue,urlcolor=darkblue,breaklinks=true}
\setlist[enumerate]{label=\textup{(\roman*)},leftmargin=2.1em,itemsep=3pt}
\newtheorem{maintheorem}{Theorem}

\newtheorem{theorem}{Theorem}[section]
\newtheorem{proposition}[theorem]{Proposition}
\newtheorem{lemma}[theorem]{Lemma}

\theoremstyle{definition}
\newtheorem{definition}[theorem]{Definition}

\theoremstyle{remark}
\newtheorem{remark}[theorem]{Remark}
\numberwithin{equation}{section}
\newcommand{\F}{\mathbb F_2}
\newcommand{\Z}{\mathbb Z}

\newcommand{\T}{\mathbb Q/\mathbb Z}
\newcommand{\im}{\operatorname{im}}
\newcommand{\Hom}{\operatorname{Hom}}

\newcommand{\Out}{\operatorname{Out}}
\newcommand{\Sq}{\operatorname{Sq}}
\newcommand{\rank}{\operatorname{rank}}
\newcommand{\ks}{\operatorname{ks}}
\newcommand{\ter}{\operatorname{ter}}
\newcommand{\Wh}{\operatorname{Wh}}
\newcommand{\TopSpin}{\mathrm{TopSpin}}

\newcommand{\topcat}{\mathrm{top}}
\newcommand{\diffcat}{\mathrm{diff}}
\newcommand{\surg}[1]{\kappa^{#1}_2}

\newcommand{\bpair}[2]{\langle #1,#2\rangle}

\title[Simple homotopy types of smooth four-manifolds]{Simple homotopy types of smooth four-manifolds \\ with finite fundamental group}

\author{Daniel Kasprowski}
\address{School of Mathematical Sciences, University of Southampton, Southampton SO17 1BJ, UK}
\email{d.kasprowski@soton.ac.uk}

\author{John Nicholson}
\address{Max Planck Institute for Mathematics, Vivatsgasse 7, 53111 Bonn, Germany}
\email{math@johnnynicholson.co.uk}


\begin{document}
\begin{abstract}
We construct closed smooth $4$-manifolds with finite fundamental group which are homotopy equivalent but not simple homotopy equivalent, even after taking a connected sum with copies of $S^2\times S^2$. Our explicit examples have remarkably large fundamental groups; in particular, the smallest examples we obtain have fundamental groups of order roughly $10^{10^{42}}$.
\end{abstract}
\maketitle

\section{Introduction}\label{sec:introduction}
A homotopy equivalence $f\colon M\to N$ between closed manifolds is \emph{simple} if its Whitehead torsion $\tau(f)\in\Wh(\pi_1N)$ vanishes, and we call two manifolds simple homotopy equivalent if there exists a simple homotopy equivalence between them. Prominent examples of odd-dimensional manifolds that are homotopy equivalent but not simple homotopy equivalent are given by lens spaces. The even-dimensional existence problem has proved more subtle. Only recently were examples constructed in every even dimension of at least four~\cite{NNP23}. These constructions give topological manifolds in dimension four and smooth manifolds in higher dimensions. In~\cite{KNV25}, we approached the smooth four-dimensional problem using equivalence relations defined up to connected sum with $S^2\times S^2$. Combining \cite[Theorem~C]{KNV25} with computations of Hambleton and Ünlü \cite[Corollary~4.9]{HU26} gives smooth examples in dimension four. The purpose of this article is to obtain smooth examples whose fundamental groups are finite.

\begin{maintheorem}\label{thm:existence}
There is a finite $2$-group $\Gamma$ and closed, oriented, smooth $4$-manifolds $M$ and $N$ with fundamental group $\Gamma$ which are homotopy equivalent but not simple homotopy equivalent.
Furthermore, $M\#k(S^2\times S^2)$ and $N\# k(S^2\times S^2)$ are not simple homotopy equivalent for all $k \ge 0$.
\end{maintheorem}

While the group $\Gamma$ in Theorem~\ref{thm:existence} is explicit, it has order $2^{\,16143(2+63\cdot2^{16137})}$. More precisely, it is obtained from the group $G_0=\operatorname{SmallGroup}(128,1377)$, i.e.\ the $1377$th group of order 128 in the GAP Small Groups Library, by the following construction. 
Let $G$ be a finite nontrivial group, let $\Pi=G*G$, and define subgroups successively by
\begin{equation}\label{eq:tower-intro}
	\begin{aligned}
		K_0&=\ker\bigl(\Pi\longrightarrow G\times G\bigr),&
		K_1&=[K_0,K_0]K_0^2,\\
		n&=|\Pi/K_1|,&
		K_2&=[K_1,K_1]K_1^n.
	\end{aligned}
\end{equation}
Here $K^m$ denotes the subgroup generated by the $m$th powers in $K$. Set $
	\Gamma(G):=\Pi/K_2$.
Then the group $\Gamma$ in Theorem~\ref{thm:existence} is $\Gamma(G_0)$. If $G$ is a $2$-group, then $\Gamma(G)$ is a $2$-group. More precisely, writing
\[
g=|G|,\qquad r_0=(g-1)^2,\qquad
n=g^2 2^{r_0},\qquad r_1=1+2^{r_0}(r_0-1),
\]
one has $|\Gamma(G)|=n^{r_1+1}$ by Lemma~\ref{lem:order}. In Remark~\ref{rem:smaller}, we obtain a substantially smaller example whose order is still roughly $10^{10^{42}}$, as mentioned in the abstract.

We now describe the proof idea of Theorem \ref{thm:existence}. 
For a group $G$, let
\[
\surg{s}(G)\colon H_2(G;\F)\longrightarrow L_4^s(\Z G),
\qquad
\surg{h}(G)\colon H_2(G;\F)\longrightarrow L_4^h(\Z G)
\]
be the degree-two components of the surgery assembly maps. Since $\surg{h}(G)$ is the composition of $\surg{s}(G)$ with the forgetful map $L_4^s(\Z G)\to L_4^h(\Z G)$ we have $\ker\surg{s}(G)\subseteq\ker\surg{h}(G)$.
We will not need the maps $\surg{h}$ and $\surg{s}$ explicitly, and refer to Hambleton--Milgram--Taylor--Williams~\cite{HMTW88} and Davis~\cite{Dav05} for more information. Hambleton--\"Unl\"u~\cite[Corollary~4.9]{HU26} prove that $G_0$ satisfies $\surg{h}(G_0)=0$ while $\surg{s}(G_0)\ne0$. In particular, $\ker\surg{s}(G_0)\subsetneq\ker\surg{h}(G_0)$. Thus Theorem~\ref{thm:existence} is implied by the following.

\begin{maintheorem}\label{thm:general}
Suppose that $G$ is finite and $\ker\surg{s}(G)\subsetneq\ker\surg{h}(G)$. Then $\Gamma(G)$ is finite and there are closed, oriented, smooth $4$-manifolds $M,N$ with fundamental group $\Gamma(G)$ which are homotopy equivalent but $M\#k(S^2\times S^2)$ and $N\# k(S^2\times S^2)$ are not simple homotopy equivalent for all $k \ge 0$.
\end{maintheorem}

The proof of Theorem~\ref{thm:general} is similar to that of \cite[Theorem~C]{KNV25}. By \cite[Theorem~A]{KNV25}, one needs a normal $1$-type $\xi_G$ with fundamental group $G$ such that the images of $\ker\surg{s}(G)$ and $\ker\surg{h}(G)$ in the corresponding $E_{2,2}^\infty$-term of the James spectral sequence differ. For this, some control over the differentials in the James spectral sequence is required.  \cite[Theorem~C]{KNV25} shows that this holds for the free product $G*G$ provided $G$ satisfies $\ker\surg{s}(G)\subsetneq\ker\surg{h}(G)$. The group $\Gamma(G)$ is close enough to $G*G$ to still have this property, while it is also close enough to $G\times G$ to be finite. 

\subsection*{Organisation and conventions}
Section~\ref{sec:bordism} recalls the relevant information on the James spectral sequence and the results from \cite{KNV25} that are used in this paper. Section~\ref{sec:finite-quotient} establishes results about the differentials in the James spectral sequence for the group $\Gamma(G)$. These are used in Section~\ref{sec:proof} to prove Theorem~\ref{thm:general}.
All manifolds are closed, connected and oriented unless stated otherwise. 

\subsection*{Acknowledgments}
We thank Ian Hambleton and Simona Vesel\'a for useful discussions. JN was supported by the Max Planck Institute for Mathematics.

\subsection*{Declaration of AI usage}
This project began in September 2025, and its overall strategy was developed without the use of AI. We later used ChatGPT (GPT-6 Astra) to generate new ideas for overcoming computational difficulties.
This led to an initial draft of this article containing Lemma~\ref{lem:finite-quotient}.
This result is central in the proof of Proposition~\ref{prop:descent} which in turn is one of the two main ingredients in the proof of Theorem~\ref{thm:general}. The current version has been extensively rewritten and is the responsibility of the authors.

\section{Background}\label{sec:bordism}

\subsection{Normal 1-types and the tertiary invariant}
Let $G$ be a finitely presented group and $w\in H^2(G;\F)$. The normal $1$-type $\xi^{\topcat}(G,w)$ is obtained by the homotopy pullback
\begin{equation}\label{eq:normal-pullback}
\begin{tikzcd}[column sep=large,row sep=large]
 B^{\topcat}_{G,w} \arrow[r] \arrow[d,"{\xi^{\topcat}(G,w)}"']
 & BG \arrow[d,"w"]\\
 B\mathrm{STop} \arrow[r,"w_2"'] & K(\F,2).
\end{tikzcd}
\end{equation}
Replacing $B\mathrm{STop}$ by $B\mathrm{SO}$ defines $\xi^{\diffcat}(G,w)$. We assume that $\xi^{\topcat}(G,w)$ and $\xi^{\diffcat}(G,w)$ are fibrations. A normal $1$-smoothing of a manifold is a $2$-connected lift of its stable normal bundle along this fibration. In particular, it specifies an identification of its fundamental group with $G$.

The James spectral sequence~\cite[Theorem~3.1.1]{Tei92} has the form
\begin{equation}\label{eq:james}
E^2_{p,q}(G,w)=H_p(G;\Omega_q^{\TopSpin})
\ \Longrightarrow\ \Omega_{p+q}(\xi^{\topcat}(G,w)).
\end{equation}
In the degrees needed below,
\begin{equation}\label{eq:low-bordism}
\Omega_0^{\TopSpin}\cong\Z,\qquad
\Omega_1^{\TopSpin}=\Omega_2^{\TopSpin}=\F,\qquad
\Omega_3^{\TopSpin}=0.
\end{equation}
The same coefficients occur in smooth spin bordism in degrees at most three. In total degree four the filtration induced by the spectral sequence is
\[
F_{0,4}\subseteq F_{2,2}\subseteq F_{3,1}\subseteq F_{4,0}
=\Omega_4(\xi^{\topcat}(G,w)),
\]
and $F_{2,2}/F_{0,4}=E^\infty_{2,2}(G,w)$. For a class $[f\colon M\to B^{\topcat}_{G,w}]$ in $F_{2,2}$, its image in this quotient is its \emph{tertiary invariant} denoted by $\ter(M)$. Its primary and secondary invariants are, by definition, zero. Note that $\ter(M)$ depends on the choice of $f$ but its image in $E^\infty_{2,2}(G,w)/\Out_w(G)$ is independent of $f$ by~\cite[Corollary~3.1.2]{Tei92}. Here, $\Out_w(G)$ denotes the group of outer automorphisms of $G$ fixing $w$.

\subsection{Differentials in the James spectral sequence}
We consider the differentials relevant for the $(2,2)$-term. No nonzero differential leaves the $(2,2)$-term and the only possible incoming differentials are
\begin{equation}\label{eq:D}
D_{G,w}:=d^2_{4,1}\colon H_4(G;\F)\longrightarrow H_2(G;\F)
\end{equation}
and
\begin{equation}\label{eq:d3}
d^3_{5,0}\colon\ker D'_{G,w}\longrightarrow
H_2(G;\F)/\im D_{G,w},
\qquad
D'_{G,w}:=d^2_{5,0}\colon H_5(G;\Z)\longrightarrow H_3(G;\F).
\end{equation}
 Consequently there is a canonical epimorphism
\begin{equation}\label{eq:rho}
\rho_w\colon H_2(G;\F)\longrightarrow E^\infty_{2,2}(G,w).
\end{equation}

Define
\[
\Sq^2_w=\Sq^2+w\smile-.
\]
The formula for the $d_2$-differential from~\cite[Theorem~3.1.3]{Tei92} gives for all $\alpha\in H^2(G;\F), \beta\in H^3(G;\F), c\in H_4(G;\F)$ and $d\in H_5(G;\Z)$ that
\begin{equation}\label{eq:d2-pairings}
	\bpair{\alpha}{D_{G,w}(c)}=\bpair{\Sq^2_w\alpha}{c} \text{ and }
	\bpair{\beta}{D'_{G,w}(d)}=\bpair{\Sq^2_w\beta}{r_2d},
\end{equation}
where $r_2$ is reduction modulo two.

Let $j\colon\F\to\T$ be the injective homomorphism with $j(1)=1/2$. Since $\T$ is an injective abelian group, the universal coefficient theorem gives a natural identification $H^5(G;\T)\cong\Hom(H_5(G;\Z),\T)$. 
The universal coefficient theorem also gives a natural identification
\[
H^3(G;\F)\cong \Hom(H_3(G;\F),\F)\cong \Hom(H_3(G;\F),\T),
\]
where the second isomorphism is given by post-composition with $j$. In these identifications,
\begin{equation}\label{eq:dual-Dprime}
	\Hom(D'_{G,w},\T)=j_*\Sq^2_w\colon H^3(G;\F)\longrightarrow H^5(G;\T).
\end{equation}
Indeed, $\Hom(D'_{G,w},\T)$ sends $\beta\in H^3(G;\F)$ to $j_*\langle \Sq_w^2\beta,r_2(-)\rangle$ by \eqref{eq:d2-pairings} which agrees with the evaluation of $j_*\Sq_w^2\beta$.

\begin{lemma}\label{lem:character-quotient}
	Restriction to $\ker D'_{G,w}$ induces a natural isomorphism
	\begin{equation*}
		\frac{H^5(G;\T)}{j_*\Sq^2_w\bigl(H^3(G;\F)\bigr)}
		\ \cong\ \Hom(\ker D'_{G,w},\T).
	\end{equation*}
\end{lemma}

\begin{proof}
	Since $\T$ is injective, the short exact sequence $0\to \ker D'_{G,w}\to H_5(G;\Z)\to \im D'_{G,w}\to 0$ yields a short exact sequence
	\[0\to \Hom(\im D'_{G,w},\T)\to \Hom(H_5(G;\Z),\T)\to \Hom(\ker D'_{G,w},\T)\to 0.\]
	Again using injectivity of $\T$, the subgroup $\Hom(\im D'_{G,w},\T)$ agrees with the image of $\Hom(D'_{G,w},\T)$ which by \eqref{eq:dual-Dprime} is the same as the image of $j_*\Sq^2_w$. Thus the restriction map $\Hom(H_5(G;\Z),\T)\to \Hom(\ker D'_{G,w},\T)$ agrees with the quotient map $H^5(G;\T)\to H^5(G;\T)/j_*\Sq^2_w\bigl(H^3(G;\F)\bigr)$ as claimed.
\end{proof}

\begin{definition}\label{def:Delta}
	For $\alpha\in H^2(G;\F)$ with $\Sq^2_w\alpha=0$, the evaluation map $\langle\alpha,-\rangle$ factors through $H_2(G;\F)/\im D_{G,w}$ by \eqref{eq:d2-pairings}, and we denote the induced map by $\overline\alpha\colon H_2(G;\F)/\im D_{G,w}\to\F$.
	Define
	\begin{equation*}
		\Delta_{G,w}(\alpha)\in
		\frac{H^5(G;\T)}{j_*\Sq^2_w\bigl(H^3(G;\F)\bigr)}
	\end{equation*}
	to be the class which corresponds to
	\[
	\ker D'_{G,w}\xrightarrow{d^3_{5,0}}
	E^3_{2,2}(G,w)\xrightarrow{~\overline\alpha~}\F
	\xrightarrow{~j~}\T
	\]
	by Lemma~\ref{lem:character-quotient}.
\end{definition}

We will use the following proposition to detect elements in $E_{2,2}^\infty(G,w)$.

\begin{proposition}\label{prop:Delta}
	For $\alpha\in H^2(G;\F)$, the following are equivalent:
	\begin{enumerate}
		\item evaluation by $\alpha$ factors through $\rho_w$;
		\item $\Sq^2_w\alpha=0$ and $\Delta_{G,w}(\alpha)=0$.
	\end{enumerate}
	Moreover, if $\Sq^2_w\alpha=0$, then for every homomorphism $f\colon G'\to G$ and $w':=f^*w$, we have
	\begin{equation}\label{eq:Delta-natural}
		f^*\Delta_{G,w}(\alpha)=\Delta_{G',w'}(f^*\alpha).
	\end{equation}
\end{proposition}

\begin{proof}
	By \eqref{eq:d2-pairings}, evaluation by $\alpha$ factors through $H_2(G;\F)/\im D_{G,w}$ if and only if $\Sq^2_w\alpha=0$. As in Definition~\ref{def:Delta}, we denote the induced map by $\overline\alpha$.
	
	In this case, $\Delta_{G,w}(\alpha)=0$ if and only if $\overline\alpha$ factors through $E_{2,2}^\infty(G,w)$ since $j$ is injective. This proves the equivalence.
	
	By naturality of $\Sq^2$, $\Sq^2_{w'}(f^*\alpha)=f^*\Sq^2_{w}(\alpha)=0$ and hence $\Delta_{G',w'}(f^*\alpha)$ is defined. The equality \eqref{eq:Delta-natural} now follows from the naturality of the James spectral sequence.
\end{proof}

\subsection{Detecting (simple) homotopy equivalences up to stabilisation}
In this section we recall the results from \cite{KNV25} that we will use to show the existence of homotopy equivalent but not simple homotopy equivalent manifolds. The next lemma will be used to obtain smooth manifolds.

\begin{lemma}[{\cite[Lemma~3.4]{KNV25}}]\label{lem:characteristic}
The map $w\colon H_2(G;\F)\to\F$ factors through $\rho_w$. If $[f\colon M\to B^{\topcat}_{G,w}]$ belongs to $F_{2,2}$, then the signature $\sigma(M)$ of $M$ is divisible by eight and
\begin{equation}\label{eq:KS}
\ks(M)\equiv \sigma(M)/8+\bpair{w}{\ter(M)}\pmod2.
\end{equation}
\end{lemma}

The following proposition is a consequence of \cite[Theorem~A]{KNV25} and will be the main ingredient in the proof of Theorem~\ref{thm:general}.

\begin{proposition}\label{prop:criterion}
Let $G$ be a finitely presented group. Suppose that $w\in H^2(G;\F)$ and $z\in H_2(G;\F)$ satisfy
\begin{equation}\label{eq:criterion}
z\in\ker\surg{h}(G),\qquad \bpair{w}{z}=0,\qquad
\rho_w(z)\notin\rho_w(\ker\surg{s}(G)).
\end{equation}
Then there are closed oriented smooth $4$-manifolds with fundamental group $G$ which are homotopy equivalent but not simple homotopy equivalent up to any stabilisations. This remains true when orientation-reversing simple homotopy equivalences are allowed.
\end{proposition}

\begin{proof}
Let $\xi=\xi^{\topcat}(G,w)$. Let $f\colon M\to B^{\topcat}_{G,w}$ be a normal $1$-smoothing with $\ter(M)=\rho_w(z)\in E^\infty_{2,2}(G,w)$. By taking a connected sum with copies of $E8$ with either orientation if necessary, we can assume that $\sigma(M)=0$. By Lemma~\ref{lem:characteristic}, $\ks(M)=\langle w,z\rangle=0$. Let $N$ be a manifold with fundamental group $G$ and a nullbordant normal $1$-smoothing over $\xi$. Again $\ks(N)=0$ by Lemma~\ref{lem:characteristic}. It follows that $M$ and $N$ are stably smoothable.

By \cite[Theorem~3.2~(iv)]{KNV25}, $M$ and $N$ are homotopy equivalent up to stabilisations by $S^2\times S^2$ since their signatures agree and the difference of their tertiary invariants is $\rho_w(z)$, where $z\in\ker \surg{h}$. Hence after stabilisations, we can assume that $M$ and $N$ are smooth, homotopy equivalent manifolds.

It remains to show $M$ and $N$ are not simple homotopy equivalent up to stabilisations by $S^2\times S^2$. By naturality of $\surg{s}$, $\ker\surg{s}\subseteq H_2(G;\F)$ is invariant under the action of $\Out_w(G)$. As mentioned above, by~\cite[Corollary~3.1.2]{Tei92}, $\ter(M)\in E^\infty_{2,2}(G,w)/\Out_w(G)$ is independent of the choice of normal $1$-smoothing. Since $\rho_w(z)\notin\rho_w(\ker\surg{s}(G))$, it follows that the difference of the tertiary invariants of $M$ and $N$ is not contained in $\rho_w(\ker\surg{s}(G))$ for any choice of normal $1$-smoothings. Thus by \cite[Theorem~A~(iii)]{KNV25}, $M$ and $N$ are not simple homotopy equivalent up to stabilisations by $S^2\times S^2$.

Applying the same argument with $-N$ instead of $N$ shows that there are no orientation-reversing simple homotopy equivalences either.
\end{proof}

\section{Controlling the differentials}\label{sec:finite-quotient}

While elementary, the following lemma is the key input in controlling the differentials in the James spectral sequence. 

\begin{lemma}\label{lem:finite-quotient}
Suppose that
\begin{equation}\label{eq:free-extension}
1\longrightarrow K\longrightarrow E\overset{p}{\longrightarrow}Q\longrightarrow1
\end{equation}
is an extension with $Q$ finite of order $q\geq 2$ and $K$ a finitely generated free group. For $m\geq2$, let
\[
K(m)=[K,K]K^m,\qquad E_m=E/K(m).
\]
Then $E_m$ is finite, and the following assertions hold.
\begin{enumerate}
\item For $d\geq2$, the map $H^d(Q;\F)\to H^d(E_2;\F)$ is zero on
\[
\ker\bigl(p^*\colon H^d(Q;\F)\to H^d(E;\F)\bigr).
\]
\item For $d\geq3$, the map $H^d(Q;\T)\to H^d(E_{q};\T)$ is zero on
\[
\ker\bigl(p^*\colon H^d(Q;\T)\to H^d(E;\T)\bigr).
\]
\end{enumerate}
\end{lemma}

\begin{proof}
The subgroup $K(m)$ is characteristic in $K$, hence normal in $E$. If $r=\rank K$, then
\[
K/K(m)\cong(\Z/m)^r,
\]
so $E_m$ is finite. Write $A$ for either $\F$ or $\T$. The cohomological Lyndon--Hochschild--Serre spectral sequence for~\eqref{eq:free-extension} has $E_2$-page $H^p(Q;H^q(K;A))$, the action of $Q$ on $H^q(K;A)$ is induced by conjugation in $E$.
Since $K$ is free, $H^q(K;A)=0$ for $q\geq2$. The kernel of the edge map $p^*\colon H^d(Q;A)\to H^d(E;A)$ in degree $d\geq2$ is consequently the image of the $d^2$-differential $\tau\colon H^{d-2}(Q;H^1(K;A))\longrightarrow H^d(Q;A)$.

The quotient map gives a commutative diagram of extensions
\[
\begin{tikzcd}[column sep=small]
	1\arrow[r]&K\arrow[r]\arrow[d]&E\arrow[r]\arrow[d]&Q\arrow[r]\arrow[d,equal]&1\\
	1\arrow[r]&K/K(m)\arrow[r]&E_m\arrow[r]&Q\arrow[r]&1.
\end{tikzcd}
\]

We now prove (i). For $A=\F$ and $m=2$, every homomorphism $K\to\F$ factors through $K/K(2)$. Hence
\[
H^1(K/K(2);\F)\overset{\cong}{\longrightarrow}H^1(K;\F)
\]
is an isomorphism of $Q$-modules. Given $x\in \ker(p^*\colon H^d(Q;\F)\to H^d(E;\F))$ with $d\geq 2$, there exists an element $y\in H^{d-2}(Q;H^1(K;\F))$ with $\tau(y)=x$. Let $y'\in H^{d-2}(Q;H^1(K/K(2);\F))$ be a preimage of $y$. By naturality of the Lyndon--Hochschild--Serre spectral sequence, the $d^2$-differential sends $y'$ to $x$. Hence $x$ is in the kernel of the map $H^d(Q;\F)\to H^d(E_2;\F)$ as claimed. This shows~(i).

For~(ii), let $B=H^1(K;\T)$. As an abelian group, $B\cong(\T)^r$, so multiplication by $m$ is surjective. Thus there is an exact sequence of $Q$-modules
\begin{equation}\label{eq:divisible}
0\longrightarrow B[m]\longrightarrow B\overset{m}{\longrightarrow}B\longrightarrow0,
\end{equation}
where $B[m]$ is defined to be the kernel of multiplication by $m$.
For a finite group $Q$ of order $q$, every cohomology group of positive degree is annihilated by multiplication by $q$. It follows that the long exact sequence in cohomology with coefficients given by \eqref{eq:divisible} with $m=q$ induces for every $j>0$ a surjection $H^j(Q;B[q])\longrightarrow H^j(Q;B)$.
Since $K/K(q)\cong K_{ab}/qK_{ab}$, restriction identifies $H^1(K/K(q);\T)$ with $B[q]$ as $Q$-modules. Taking $j=d-2>0$ above therefore
gives a surjection $H^{d-2}(Q;H^1(K/K(q);\T))\to H^{d-2}(Q;H^1(K;\T))$. We now argue as in the first part. Given $x\in \ker(p^*\colon H^d(Q;\T)\to H^d(E;\T))$ with $d\geq 3$, there exists an element~$y\in H^{d-2}(Q;H^1(K;\T))$ with $\tau(y)=x$. Let $y'\in H^{d-2}(Q;H^1(K/K(q);\T))$ be a preimage of~$y$. By naturality of the Lyndon--Hochschild--Serre spectral sequence, the $d^2$-differential sends~$y'$ to $x$. Hence $x$ is in the kernel of the map $H^d(Q;\T)\to H^d(E_q;\T)$ as claimed. This shows~(ii).
\end{proof}

Fix a finite nontrivial group $G$. We will now construct a map $E^\infty_{2,2}(\Gamma(G),w)\to \F$ for a suitable $w$ which we will use to show that $\rho_w(z)\notin \rho_w(\ker\surg{s}(\Gamma(G)))$ when applying Proposition~\ref{prop:criterion}. We start with some notation and two intermediate lemmas.

Fix an element $\lambda\in H^2(G;\F)$. Let
\[
\Pi=G*G,
\qquad Q_0=G\times G,
\qquad K_0=\ker(\Pi\to Q_0),
\]
and use~\eqref{eq:tower-intro} to define $K_1$, $Q_1=\Pi/K_1$, $n=|Q_1|$, and $\Gamma=\Gamma(G)$. There are quotient maps
\begin{equation}\label{eq:tower}
\Pi\overset{q}{\longrightarrow}\Gamma
\overset{v}{\longrightarrow}Q_1
\overset{u}{\longrightarrow}Q_0.
\end{equation}
By abuse of notation, for any group in this diagram, use its two maps $p_i$ to $G$ to define
\begin{equation}\label{eq:twists}
\lambda_i=p_i^*\lambda,
\qquad w=\lambda_1+\lambda_2.
\end{equation}

\begin{lemma}\label{lem:free-kernels}
The subgroups $K_0$ and $K_1$ are finitely generated free groups. The subgroups $K_1$ and $K_2$ are normal in $\Pi$, and $Q_1$ and $\Gamma$ are finite. The two factor maps $G\to\Pi$ induce embeddings into every group in~\eqref{eq:tower}.
\end{lemma}

\begin{proof}
By the Kurosh subgroup theorem, subgroups of $\Pi=G*G$ are free products of free groups and conjugates of subgroups of the factors. Each factor, and hence each of its conjugates, maps injectively to $G\times G$. Hence $K_0$ is a free group as claimed. It has finite index in the finitely generated group $\Pi$, so it is finitely generated.

The quotient $K_0/K_1$ is a finite elementary abelian $2$-group. Thus $Q_1$ is finite, and $K_1$ is a finite-index subgroup of $K_0$, hence finitely generated and free. Likewise, $K_1/K_2$ is finite, so $\Gamma$ is finite.
Since $K_0$ is normal in $\Pi$, and $K_1$ and $K_2$ are characteristic in $K_0$ and $K_1$, respectively, both are normal in $\Pi$.  Finally, the composite of each factor map with its coordinate projection $p_i$ is the identity on $G$. The induced factor maps are therefore embeddings.
\end{proof}

\begin{lemma}\label{lem:free-product}
For the class $w\in H^2(\Pi;\F)$ in~\eqref{eq:twists}, evaluation by $\lambda_1$ factors through $E^\infty_{2,2}(\Pi,w)$. 
\end{lemma}

\begin{proof}
The proof of this lemma is similar to the proof of \cite[Theorem~C]{KNV25}.
	
We have $H^2(\Pi;\F)\cong H^2(G;\F)\oplus H^2(G;\F)$ and $w=(\lambda,\lambda)$ under this isomorphism. Naturality of the James spectral sequence with respect to the factor inclusions shows that the relevant $d_2$- and $d_3$-differentials for $(\Pi,w)$ are the direct sums of two copies of the $d_2$- and $d_3$-differentials for $(G,\lambda)$, respectively. Since $\lambda_1$ is trivial on the second factor, it thus suffices to show that evaluation by $i_1^*\lambda_1$ factors through $E^\infty_{2,2}(G,\lambda)$, where $i_1\colon G\to \Pi$ is the inclusion of the first factor. This is true by Lemma~\ref{lem:characteristic} since $i_1^*\lambda_1=i_1^*p_1^*\lambda=\lambda$.
\end{proof}

Lemma~\ref{lem:finite-quotient} can now be used to show the following. Together with Proposition~\ref{prop:criterion} it will be used in the next section to prove Theorem~\ref{thm:general}.

\begin{proposition}\label{prop:descent}
For every $\lambda\in H^2(G;\F)$, evaluation by the class $\lambda_1\in H^2(\Gamma(G);\F)$ defined in~\eqref{eq:twists} induces a map
\begin{equation*}
\widehat\lambda_1\colon E^\infty_{2,2}(\Gamma(G),w)\longrightarrow\F.
\end{equation*}
\end{proposition}

\begin{proof}
By Proposition~\ref{prop:Delta}, we have to show that $\Sq^2_w\lambda_1$ and $\Delta_{\Gamma,w}(\lambda_1)$ vanish. 

We first consider $\Sq^2_w\lambda_1$. In $Q_0=G\times G$, we have $\Sq^2_w\lambda_1=\lambda_1^2+(\lambda_1+\lambda_2)\lambda_1
=\lambda_1\lambda_2$.
The class $\lambda_1\lambda_2$ pulls back trivially to $\Pi=G*G$. Apply Lemma~\ref{lem:finite-quotient}(i) in degree four to
\[
1\longrightarrow K_0\longrightarrow\Pi\longrightarrow Q_0\longrightarrow1.
\]
Its prescribed quotient is exactly $Q_1=\Pi/[K_0,K_0]K_0^2$. Consequently $\Sq^2_w\lambda_1=0\in H^4(Q_1;\F)$. By naturality of $\Sq^2_w$, also $\Sq^2_w\lambda_1=0\in H^4(\Gamma;\F)$.
Thus $\Delta_{Q_1,w}(\lambda_1)$ and $\Delta_{\Gamma,w}(\lambda_1)$ are defined.

Write $r\colon\Pi\to Q_1$ for the quotient map, and choose a class $o\in H^5(Q_1;\T)$ representing $\Delta_{Q_1,w}(\lambda_1)$. By Lemma~\ref{lem:free-product}, evaluation by $\lambda_1$ factors through $E_{2,2}^\infty(\Pi,w)$. By Proposition~\ref{prop:Delta}, it follows that $\Delta_{\Pi,w}(\lambda_1)=r^*\Delta_{Q_1,w}(\lambda_1)$ is trivial. Therefore
$r^*o=j_*\Sq^2_w(a)$ for some $a\in H^3(\Pi;\F)$.

Since $H^3(Q_0;\F)\to H^3(\Pi;\F)$ is surjective, so is $H^3(Q_1;\F)\to H^3(\Pi;\F)$. Choose an element $\widetilde a\in H^3(Q_1;\F)$ with $r^*\widetilde a=a$ and let $o':=o-j_*\Sq^2_w(\widetilde a)$.
Then $o'$ is a representative of $\Delta_{Q_1,w}(\lambda_1)$ with $r^*o'=0\in H^5(\Pi;\T)$.

Apply Lemma~\ref{lem:finite-quotient}(ii) in degree five to
\[
1\longrightarrow K_1\longrightarrow\Pi\longrightarrow Q_1\longrightarrow1.
\]
Its prescribed quotient is $\Pi/[K_1,K_1]K_1^{|Q_1|}=\Gamma$ and it follows that $v^*o'=0\in H^5(\Gamma;\T)$. Naturality of $\Delta$ implies that $v^*o'$ is a representative of $\Delta_{\Gamma,w}(\lambda_1)$. Hence it is trivial. Thus the statement follows from Proposition~\ref{prop:Delta} as claimed.
\end{proof}

\section{Proof of Theorem B and the order of $\Gamma(G_0)$}\label{sec:proof}

\begin{proof}[Proof of Theorem~\ref{thm:general}]
	Choose $x\in\ker\surg{h}(G)\setminus\ker\surg{s}(G)$ and a homomorphism $f\colon H_2(G;\F)\to \F$ such that $f(x)=1$ while $\ker\surg{s}(G)\subseteq \ker f$. Let $\lambda\in H^2(G;\F)$ be a representative of $f$. 

Let $\Gamma=\Gamma(G)$ and use the classes $\lambda_1,\lambda_2,w$ of~\eqref{eq:twists}. Write $i_1,i_2\colon G\to\Gamma$ for the factor inclusions, and define $z=(i_1)_*x+(i_2)_*x\in H_2(\Gamma;\F)$. We will now show that Proposition~\ref{prop:criterion} can be applied to $z$.

Naturality of $\surg{h}$ gives $z\in\ker\surg{h}(\Gamma)$.
Moreover, $p_i i_i=\mathrm{id}_G$ and $p_i i_j$ is trivial for $i\ne j$. Hence
$\bpair{\lambda_1}{z}=\bpair{\lambda_2}{z}=1$, and $\bpair{w}{z}=0$.

For every $y\in\ker\surg{s}(\Gamma)$, naturality of $\surg{s}$ gives $(p_1)_*y\in\ker\surg{s}(G)$. Then \[\langle \lambda_1,y\rangle=\langle \lambda,(p_1)_*y\rangle=f((p_1)_*y)=0.\]

By Proposition~\ref{prop:descent}, evaluation by $\lambda_1$ factors through $\rho_w$. The induced map $\widehat\lambda_1$ satisfies $\widehat\lambda_1(\rho_w(z))=\langle \lambda_1,z\rangle=1$ and $\widehat\lambda_1(\rho_w(y))=\langle \lambda_1,y\rangle=0$ for every $y\in \ker\surg{s}(\Gamma)$. 
In particular, $\rho_w(z)\notin\rho_w(\ker\surg{s}(\Gamma))$. Thus Theorem~\ref{thm:general} follows from Proposition~\ref{prop:criterion}.
\end{proof}
It remains to show the following lemma promised in the introduction.
\begin{lemma}
	\label{lem:order}
	Let $G$ have order $g$ and set $r_0=(g-1)^2, n=g^2 2^{r_0}$, and $r_1=1+2^{r_0}(r_0-1)$. Then $\Gamma(G)$ has order $n^{r_1+1}$.
\end{lemma}
\begin{proof}
	The kernel $K_0$ of $G*G\to G\times G$ is freely generated by the commutators of any non-trivial element in the first factor with any non-trivial element in the second factor. Hence it has rank $(g-1)^2$. 

It follows that $|Q_1|=|G\times G|\,|K_0/K_1|=g^2 2^{r_0}=n$. By the Nielsen-Schreier theorem 
\[\rank K_1=1+2^{r_0}(r_0-1)=1+n\left(1-\frac2g\right)=r_1.\]

Finally $K_1/K_2\cong(\Z/n)^{r_1}$, so
\begin{equation*}
|\Gamma(G)|=|Q_1|\,|K_1/K_2|=n^{r_1+1}.\qedhere
\end{equation*}
\end{proof}
We can replace one factor of $G$ in $G*G$ by $C_2=\Z/2\Z$ to obtain the following smaller examples.
\begin{remark}\label{rem:smaller}
	For $G$ as in Theorem~\ref{thm:general}, replace $G*G$ by $\Pi'=G*C_2$ and $G\times G$ by $G\times C_2$ in \eqref{eq:tower-intro}. Use primes for the resulting subgroups and quotients, and write $\Gamma'(G)=\Pi'/K'_2$.
	Choose $x,\lambda$ as in the proof of Theorem~\ref{thm:general}, and nonzero classes $\mu\in H^2(C_2;\F)$ and $\eta\in H_2(C_2;\F)$.
	With the analogous projections and factor inclusions, let
	\[
	w=p_G^*\lambda+p_C^*\mu,\qquad z=(i_G)_*x+(i_C)_*\eta.
	\]
	By \cite[Theorem A (iii)]{HU26}, $\kappa_2^s(C_2)=0$, so $z\in\ker\kappa_2^h(\Gamma'(G))$ and $\langle w,z\rangle=0$.
	
	Evaluation by $p_G^*\lambda$ factors through $E^\infty_{2,2}(\Pi',w)$ as in Lemma~\ref{lem:free-product}: its restrictions to the factors are $\lambda$ and zero.
	The proof of Proposition~\ref{prop:descent} applies with $G\times C_2$ in place of $G\times G$, since $(p_G^*\lambda)(p_C^*\mu)$ vanishes on $\Pi'$ and
	$H^3(G\times C_2;\F)\to H^3(\Pi';\F)$ is surjective.
	Naturality under $p_G$ shows that the resulting map on $E^\infty_{2,2}$ annihilates $\rho_w(\ker\kappa_2^s(\Gamma'(G)))$, whereas its value on $\rho_w(z)$ is one. Proposition~\ref{prop:criterion} therefore applies.
	
	The ranks and orders are now
	\[
	r'_0=g-1,\qquad n'=2g\,2^{g-1},\qquad r'_1=1+2^{g-1}(g-2),
	\]
	and $|\Gamma'(G)|=(n')^{r'_1+1}$. For $G=G_0$, this gives
	\[
	|\Gamma'(G_0)|=2^{135(2+63\cdot2^{128})}\approx 10^{10^{41.94}}.
	\]
	We do not claim that these examples are minimal.
\end{remark}

\bibliographystyle{amsalpha}
\bibliography{classification}
\end{document}